\documentclass[11pt,a4paper,oneside,english]{article}

\usepackage[utf8]{inputenc}
\usepackage[T1]{fontenc}
\usepackage[english]{babel}
\usepackage{graphicx}
\usepackage{csquotes}
\usepackage{amsmath}
\usepackage{amssymb}
\usepackage{amsthm}
\usepackage{algorithm}
\usepackage{algpseudocode}
\usepackage{tikz}
\usepackage{authblk}
\usepackage{subcaption}
\usepackage{caption}
\usepackage[bottom]{footmisc}

\usepackage[top=2cm, bottom=2cm, left=2cm, right=2cm]{geometry}

\theoremstyle{plain}
\newtheorem{theorem}{Theorem}[section] 
\theoremstyle{definition}
\newtheorem{definition}[theorem]{Definition}
\newtheorem{remark}[theorem]{Remark}
\newtheorem{form}[theorem]{Formulation}

\newcommand{\spfhat}[2]{\ensuremath{ \left( #1 \text{ , } #2 \right)_{\hat{\Omega}_f}}}
\newcommand{\spshat}[2]{\ensuremath{ \left( #1 \text{ , } #2 \right)_{\hat{\Omega}_s}}}

\begin{document}

	\selectlanguage{english}
	\pagenumbering{arabic}

	\title{Parallel Block-Preconditioned Monolithic Solvers for Fluid-Structure-Interaction Problems}
	\author[1]{D. Jodlbauer}
	\author[2]{U. Langer}
	\author[3]{T. Wick}

	\affil[1]{Doctoral Program on Computational Mathematics, Johannes Kepler University, Altenbergerstr. 69, A-4040 Linz, Austria}
	\affil[2]{Johann Radon Institute for Computational and Applied Mathematics, Austrian Academy of Sciences, Altenbergerstr. 69, A-4040 Linz, Austria}
	\affil[3]{Institut f\"ur Angewandte Mathematik, Leibniz Universit\"at Hannover, Welfengarten 1, 30167 Hannover, Germany}

	\date{}

	\maketitle

	\begin{abstract}
		In this work, we consider the solution of fluid-structure interaction problems
		using a monolithic approach for the coupling between fluid and solid
		subproblems. The coupling of both equations is realized
		by means of
		the
		arbitrary Lagrangian-Eulerian framework and a nonlinear harmonic mesh motion model.
		Monolithic approaches require the solution of large, ill-conditioned linear systems of algebraic equations at every Newton step. Direct solvers tend to use too much memory even for a relatively small number of degrees of freedom,
		and, in addition, exhibit superlinear grow in arithmetic complexity.
		Thus, iterative solvers are the only viable option.
		To ensure convergence of iterative methods within a reasonable amount of iterations,
		good and, at the same time, cheap preconditioners have to be developed.
		We study physics-based block preconditioners,
		which are derived from the block $LDU$-factorization of the FSI Jacobian,
		and their performance on
		distributed memory parallel computers in terms of two- and three-dimensional test cases
		permitting large deformations.
	\end{abstract}

	\section{Introduction}

	Fluid-structure interaction problems (FSI) are important in many
	technical and life science applications.
	Air flow around the wings of an aircraft
	or the flow through rotating turbine blades
	are two of such  typical
	technical examples, see, e.g., \cite{Hsu2012,PiFa01}.
	In hemodynamics, the numerical simulation of the vascular blood flow  is another
	FSI example where the interactions between blood flow and the
	walls of the vessels must be taken into account. The simulation of the human heart
	and the analysis of aneurysms are important medical applications,
	see, e.g., \cite{FormaggiaQuarteroniVeneziani:2009} and \cite{BazilevsEtAl:2010}, respectively.
	More FSI applications are compiled in various books \cite{BuSc06,FoQuaVe09,GaRa10,BaTaTe13,BoGaNe2014,BazilevsTakizawa:2016,Ri17_fsi,FrHoRiWiYa17}.

	The traditional approach to the solution of FSI problems
	makes use of available fluid and solid solvers in
	an alternate iteration between fluid and solid,
	and an information exchange across the interface
	via the interface conditions.
	This class of solvers are called
	{\it partitioned solvers},
	see, e.g.,
	\cite{MokWa01,PiFa01,OldAddedMass,SaGeBou06,WaGeRa07,VieDuVe08,BaNoVe08,DeBrHaeVie08,DeHaeAnnBrVie10,HuiLa14a,LangerYang:2015a}
	and the references therein
	for recent developments of partitioned methods.
	Beside the advantage of using available solvers,
	often provided by different codes,
	there are several disadvantages of partitioned solvers.
	These disadvantages are connected with the high complexity,
	loss of robustness due to so-called added-mass effects,
	the difficulties connected with the error control
	of the fluid and solid iterations, and, last but not least,
	the sequentiality of the alternating iteration process.

	These drawbacks are the main reasons why  {\it monolithic solvers}
	have attracted more and more attention during the last decade;
	see, e.g.,
	\cite{Heil2004,TuHrBenchmark,BuSc06,FSI_Book,BaQuaQua08,Brummelen2008_solver,FSI:AMG,BaTaTe13,
		LaYa2016,LangerYang:2018a,Richter:GMG}.
	There are different approaches to construct monolithic solvers for
	the linear FSI system that arises at every linearization step.
	Monolithic geometric and algebraic multigrids can be used
	as solvers or preconditioners in connection with Krylov
	space methods like GMRES (generalized minimal residuals),
	see
	\cite{HrTu06a,RazzaqDamanikHronOuazziTurek:2012,Richter:GMG}
	and
	\cite{FSI:AMG,MayrKloeppelWallGee:2015,LangerYang:2018a}, respectively.
	Likewise domain decomposition methods can be exploited
	\cite{BarkerCai:2010,WuCai:2014,CNM:CNM2756}.
	Another starting point for deriving efficient preconditioners for Krylov subspace solvers
	is the block  $LDU$-factorization of the linearized FSI matrix.
	Different arrangements of the blocks and different approximations
	of the blocks in the factorization lead to
	different
	inexact block  $LDU$-factorizations
	that can serve as preconditioners in Krylov subspace solvers,
	see
	\cite{Heil2004,HeilHazelBoyle:2008,
		BaQuaQua08,BadiaQuainiQuarteroni:2008,
		ParallelFSI,Hemodynamics,
		CNM:CNM2756,
		MuddleMihajlovicHeil:2012,
		LaYa2016,LangerYang:2018a}
	and the references cited there.
	In the engineering community, this class of preconditioners
	are also called physics-based block preconditioners.
	One important advantage of physics-based block preconditioners is their modularity
	that allows the reuse of available solid, fluid and mesh movment (elliptic)
	solvers
	similar to the partitioned approach, but now as a part of the preconditioner.

	There are many publications on FSI problems with exciting applications
	from different areas, but
	there are to date only a few publications studying the parallel
	performance of FSI solvers.
	In \cite{BarkerCai:2010}, an overlapping domain decomposition
	(additive Schwarz) preconditioner for the GMRES solver that provides an inexact
	solve of the Jacobian system at each Newton step is proposed.
	This parallel solution technique, which was developed and tested
	for two-dimensional FSI problems in this  paper,
	has been extended to three dimensions in  \cite{WuCai:2014},
	see also  \cite{KongCai:2016}.
	This Newton-Krylov-Schwarz FSI solver shows a very good
	parallel perfomance, at least, for the examples studied in
	these papers. All numerical tests were performed for
	characteristic tube-like geometries typically arising in blood flow
	simulations and under the assumption of
	geometric and material linear elastic behavior of the solid,
	whereas the fluid is assumed to be a Newtonian fluid.
	For the same class of FSI problems, the parallel performance of physics-based block
	preconditioners were investigated in \cite{ParallelFSI}.
	Strong and weak scalability studies were presented for a moderate number of cores.
	In the very recent publications \cite{DeparisFortiGrandperrinQuarteroni:2016}
	and \cite{FortiQuarteroniDeparis:2017}, these studies have been extended
	to a larger number of cores, again including scalability tests, and to a nonconforming fluid-structure coupling
	via the internodes technique proposed in \cite{DeparisFortiGervasioQuarteroni:2016}.

	In this paper, we follow basically the developments
	of \cite{LaYa2016,LangerYang:2018a,JodlbauerWick:2017}
	to construct a monolithic
	GMRES solver
	preconditioned by a physics-based block preconditioner.
	The major novelty is the extension to high performance computing
	and a fully parallelized programming code. This code is
	dimension-independent and can simulate two-dimensional
	as well as
	three-dimensional configurations. In 2d, we consider
	a challenging FSI benchmark problem \cite{TuHrBenchmark,BuSc06} with
	large solid deformations. In 3d, we adopt another benchmark-like
	configuration \cite{Richter:GMG}. The main goals are to
	verify the functionals of interest to show that our modeling
	yields the correct physical results,
	and scalability tests in order
	to show the performance of the parallelization.

	The remainder of this paper is organized as follows:
	In Section~\ref{sec:fsi}, we recall the governing FSI equation
	in the Arbitrary Lagrangian-Eulerian (ALE) setting.
	Its line variational formulation, time discretization, Newton linearization,
	and spacial discretization are shortly described in Section~\ref{sec:discrete}.
	Section~\ref{sec:pc} introduces the approximate block-LDU-preconditioner
	that preconditions the GMRES solver that is used to solve the huge
	system of algebraic equations arising at each Newton step.
	The parallel implementation is described in Section~\ref{sec:parallel}.
	In Section~\ref{sec:results}, we present our numerical results for
	two- and three-dimensional benchmark problems including convergence
	and parallel performance studies.
	Finally, we draw our conclusions.

	\section{FSI Equations}
	\label{sec:fsi}

	We denote
	the computational domain of the
	fluid-structure interaction problem
	by $\Omega \subset \mathbb{R}^d$, $d=2,3$.
	This domain is supposed to be
	time independent but consists of two time dependent subdomains
	$\Omega_f (t)$ and $\Omega_s (t)$ with a moving interface $\Gamma_i
	(t) = \partial\Omega_f (t) \cap \partial\Omega_s (t)$. The initial (or
	later reference) domains are denoted by $\hat\Omega_f$ and
	$\hat\Omega_s$, respectively, with the interface
	$\hat\Gamma_i$. Further, we denote the outer boundary with
	$\partial\hat\Omega = \hat\Gamma = \hat\Gamma_D \cup \hat\Gamma_N$
	where $\hat\Gamma_D $ and $\hat\Gamma_N $ denote Dirichlet and Neumann
	boundaries, respectively. Our philosophy is to transform and solve all
	equations in $\hat\Omega$. A prototype setting of an FSI configuration is
	displayed in Figure \ref{fsi:geometry} in Section \ref{sec:results}.
	In addition to the spatial domains,
	we introduce the time interval $I:= (0,T]$ where $T>0$ is the end time value.

	Let us now briefly recall the governing equations for fluid-structure interaction.
	Usually, fluid and solid equations are modelled in different coordinate
	systems, namely Eulerian and Lagrangian coordinates. In order to couple those
	equations on a common interface, it is necessary to use a common coordinate
	system. A popular choice within the FSI framework is the (arbitrary) extension
	of the Lagrange coordinates from the solid domain into the fluid domain, also
	called ALE-coordinates \cite{HuLiZi81,DoGiuHa82}. Such an extension is given
	via the solution of an additional auxiliary problem, e.g. a harmonic extension
	in our case. A comparison of various extension methods for fluid-structure interaction
	problems can be found in \cite{TW:MMPDE}. This leads to the following system
	of equations \cite{HrTu06a,DuRaRi09,TW:MMPDE}:

	\begin{form}[Strong form of FSI equations]
		Find $\hat v_f:I\times\hat\Omega_f\to \mathbb{R}^d, \hat p_f:I \times \hat\Omega_f\to \mathbb{R}, \hat u_f:I \times \hat\Omega_f\to \mathbb{R}^d,\hat v_s: I \times \hat\Omega_f\to \mathbb{R}^d, \hat u_s:I \times \hat\Omega_s\to \mathbb{R}^d$:

		\begin{equation}
		\label{eq:fsi:strong}
		\begin{split}
		\hat{J} \hat{\rho}_f \hat{\partial}_t \hat{v}_f + \hat{J} \hat{\rho}_f \hat{\nabla} \hat{v}_f \hat{F}^{-1} \cdot (\hat{v}_f - \partial_t \hat{\mathcal{A}}) - \text{div}_R (\hat{J} \hat\sigma_f \hat{F}^{-T}) & = 0
		\\
		\hat{J} \ \text{tr}(\hat{\nabla} \hat{v}_f \hat{F}^{-1}) & = 0
		\\
		\\
		\hat{\rho}_s \partial_t \hat{v}_s - \text{div}_R (\hat{F} \hat{\Sigma}) & = 0
		\\
		\left( \partial_t \hat{u}_s - \hat{v}_f \right) & = 0
		\\
		\\
		- \text{div}_R (\frac{1}{\hat J}\nabla \hat{u}_f) & = 0.
		\end{split}
		\end{equation}
	\end{form}
	In the rest of this section, we explain all equations, terms and variables
	in more detail.
	The first set of equation describe the incompressible Navier-Stokes equations,
	the second set of equation describe nonlinear elastodynamics, and in the third
	set, a nonlinear harmonic mesh motion model \cite{StTeBe03} is adopted
	since we are interested in modeling large solid deformations.

	The interface and boundary conditions are given by the equations

	\begin{alignat}{3}
	\hat{J} \hat{\sigma}_f \hat{F}^{-T} \hat{n}_f
	+ \hat{F} \hat{\Sigma}_s & = 0 \text{ on }
	I \times \hat{\Gamma}_I \hspace{10mm}
	& \hat{v}_f & = 0 \text{ on } I\times \hat{\Gamma}_{top} \cup
	I \times \hat{\Gamma}_{bottom} \cup I \times \hat{\Gamma}_c
	\label{eq:fsi:strong:normal:stresses}
	\\
	\hat{v}_f - \hat{v}_s & = 0 \text{ on }
	I \times \hat{\Gamma}_I & \hat{v}_f & = g \text{ on }
	I \times \hat{\Gamma}_{in}
	\\
	\hat{u}_f - \hat{u}_s & = 0 \text{ on }
	I \times \hat{\Gamma}_I & \hat{u}_f & = 0 \text{ on }
	I \times \{\hat{\Gamma}_{top} \cup
	\hat{\Gamma}_{bottom} \cup
	\hat{\Gamma}_{in} \cup
	\hat{\Gamma}_{out} \cup \hat{\Gamma}_c\}
	\\
	& & \hat{u}_s & = 0 \text{ on } I \times \hat{\Gamma}_{cf}
	\\
	& & \hat{v}_s & = 0 \text{ on } I \times \hat{\Gamma}_{cf}
	\end{alignat}
	and
	\begin{equation}
	\label{eq:donothing}
	\hat{J} (-\hat{p}_f \hat{I}
	+ \hat{\rho}_f \hat{\nu}_f \hat{\nabla} \hat{v}_f \hat{F}^{-1}) \hat{F}^{-T}
	= 0 \text{ on } I\times \hat{\Gamma}_{out}.
	\end{equation}
	Thus, the flow regime is driven by the prescribed velocity
	profile $g$ at the inflow boundary $\hat{\Gamma}_{in}$.
	On the outflow boundary $\hat{\Gamma}_{out}$,
	the do-nothing condition \eqref{eq:donothing} are given  \cite{HeRaTu96},
	whereas homogeneous boundary conditions are stated otherwise.
	The initial conditions for the displacements and velocities are assumed to
	be homogeneous.

	The first equation in \eqref{eq:fsi:strong} is the incompressible Navier-Stokes system in ALE-coordinates with the ALE-mapping defined as following:
	\begin{equation}
	\hat{u}(t,\hat{x}) :=
	\begin{cases}
	\hat{u}_s(t,\hat{x}) & \hat{x} \in \hat{\Omega}_s \\
	\hat{u}_f(t,\hat{x}) & \hat{x} \in \hat{\Omega}_f
	\end{cases}
	\hspace{10mm}
	\text{ and }
	\hspace{10mm}
	\hat{\mathcal{A}}(t, \hat{x}) := \hat{x} + \hat{u}(t, \hat{x}) \text{ for } \hat{x} \in \hat\Omega,
	\end{equation}
	followed by the equations for the solid. Of course in the solid domain the
	ALE-mapping
	is nothing else than the standard coordinate transformation between Lagrangian
	and Eulerian variables.

	The last equation in \eqref{eq:fsi:strong} defines the nonlinear harmonic extension of the solid-displacement into the fluid-domain, yielding the additional mesh-motion variable given by $\hat{u}_f$. The solid displacement and velocity are denoted by $\hat{u}_s$ and $\hat{v}_s$.
	The fluid velocity and pressure in ALE-coordinates are given by $\hat{u}_f, \hat{p}_f$.

	The term $\hat{F} = \hat{I} + \hat{\nabla}\hat{u}$ denotes the gradient of the ALE mapping,
	and its
	determinant is given by $\hat{J}$. The solid stress tensor is chosen
	according to the Saint Venant Kirchhoff material law (STVK) as
	\[
	\hat{\Sigma} = 2 \mu \hat{E} + \lambda \text{tr}(\hat{E}) \hat{I},
	\]
	with the strain tensor $\hat{E}$ given by $\hat{E}
	= \frac{1}{2}(\hat{\nabla} \hat{u} + \hat{\nabla} \hat{u}^T
	+ \hat{\nabla} \hat{u}^T \hat{\nabla} \hat{u})$ and the well-known Lam\'{e}
	parameters $\lambda$ and $\mu$.
	The fluid stress tensor is given by
	\[
	\hat{\sigma}_f = -\hat{p}_f \hat{I} + \hat{\rho}_f \hat{\nu}_f
	(\hat{\nabla} \hat{v}_f \hat{F}^{-1} + \hat{F}^{-T} \hat{\nabla} \hat{v}_f ^
	T).
	\]
	Here, $\hat I$ is the identity matrix. Moreover, the
	kinematic viscosity of the fluid is denoted by $\hat{\nu}_f$.
	Material densities of fluid and solid are denoted by $\hat{\rho}_f$ and $\hat{\rho}_s$, respectively.
	Interface conditions consist of the coupling of stresses, which propagate forces from one subdomain into the other, as well as the continuity of velocities and displacements resulting from the no-slip condition and the definition of the ALE-extension.

	\section{{Variational Formulation and Discretization}}
	\label{sec:discrete}

	Within the next subsections, we discuss the usual steps to
	derive
	a discretized version of the nonlinear, time-dependent FSI system \eqref{eq:fsi:strong} - \eqref{eq:donothing}.

	\subsection{Line Variational Formulation}

	As usual the finite element discretization starts from a variational
	formulation of our problem. Special care has to be taken on the
	interface when defining the test and
	trial functions.
	In order to circumvent Bochner function spaces, we define the
	equations on the time-space continuous level for almost all times and
	only specify the spatial spaces in more detail.
	For almost all times $t$, we seek
	\begin{alignat*}{2}
	( \hat{u}_s, \hat{u}_f ) \in V_u &:= \{ \hat{u}_s \in V_u^s, \hat{u}_f \in H^1(\hat{\Omega}_f)^{d}: \
	&&\hat{u}_s = 0 \text{ on } \partial \hat{\Omega}_s \backslash \hat{\Gamma}_I,
	\hat{u}_f = 0 \text{ on } \partial \hat{\Omega}_f \backslash \hat{\Gamma}_I,
	\hat{u}_s = \hat{u}_f \text{ on } \hat{\Gamma}_I \},
	\\
	( \hat{v}_s, \hat{v}_f ) \in V_v &:= \{ \hat{v}_s \in V_v^s, \hat{v}_f \in H^1(\hat{\Omega}_f)^{d}: \
	&&\hat{v}_s = 0 \text{ on } \partial \hat{\Omega}_s \backslash \hat{\Gamma}_I,
	\hat{v}_f = 0 \text{ on } \partial \hat{\Omega}_f \backslash (\hat{\Gamma}_I \cup \hat{\Gamma}_{in}),
	\\
	& &&\hat{v}_f = g(t) \text{ on } \hat{\Gamma}_{in},
	\hat{v}_s = \hat{v}_f \text{ on } \hat{\Gamma}_I
	\}
	\end{alignat*}
	and $\hat{p}_f \in V_p := L^2(\hat{\Omega}_f)$.
	The vector-valued spaces $V_u^s$ and $V_v^s$ denote the respective function spaces for the solid displacement and velocity, taking into account the non-linear structure of the equations, see, e.g., \cite{Ciarlet90,Ball76}.
	The test spaces are similar to the trial spaces:
	for the velocity fields, we take
	\[
	( \hat{\varphi}^v_s, \hat{\varphi}^v_f ) \in V_v^0 := V_v^0 := V_v \text{ with } g(t) \equiv 0.
	\]
	For the displacements, we employ the modification
	$\hat\varphi^u_v = 0$ on  $\hat\Gamma_I$.
	This avoids a non-physical coupling from the fluid mesh back to the solid displacements. The pressure test space coincides with its trial space.

	This leads to the variational formulation of our FSI-equations. Note that the interface terms vanish because of the continuity of stresses and the respective test-functions.

	\begin{form}[Line variational formulation of FSI-ALE]
		Find $\left( (\hat{u}_s, \hat{u}_f), (\hat{v}_s, \hat{v}_f), \hat{p}_f \right)$ in $V_u \times V_v \times V_p$ such that the following variational equations are satisfied for almost all times $t \in I$: 

		\begin{equation}
		\label{eq:weak:FSI}
		\begin{split}
		\spfhat{\hat{J} \hat{\rho}_f \hat{\partial}_t \hat{v}_f}{\hat{\varphi}^v_f}
		+ \spfhat{\hat{J} \hat{\rho}_f \hat{\nabla} \hat{v}_f \hat{F}^{-1} \cdot ( \hat{v}_f - \hat{w} ) }{\hat{\varphi}^v_f}
		+ \spfhat{\hat{J} \hat{F}^{-T} \hat{\sigma}_f}{\hat{\nabla} \hat{\varphi}^v_f} &
		\\
		- \left\langle \hat{\rho}_f \hat{\nu}_f \hat{J} \hat{F}^{-T} \hat{\nabla} \hat{v}_f^T \hat{F}^{-T} \cdot \hat{n}_f \ , \hat{\varphi}^v_f \right\rangle_{\hat{\Gamma}_{out}}	& = 0
		\\
		\spfhat{\hat{J} \ \text{tr}(\hat{\nabla} \hat{v}_f \hat{F}^{-1})}{\hat{\varphi}^p_f} & = 0
		\\
		\\
		\spshat{\hat{\rho}_s \hat{\partial}_t \hat{v}_s}{\hat{\varphi}^v_s}
		+ \spshat{\hat{F} \hat{\Sigma}_s}{\hat{\nabla} \hat{\varphi}^v_s} & = 0
		\\
		\spshat{\hat{\partial}_t \hat{u}_s - \hat{v}_s}{\hat{\varphi}^u_s} & = 0
		\\
		\\
		\spfhat{\dfrac{1}{\hat{J}} \hat{\nabla} \hat{u}_f}{\hat{\nabla} \hat{\varphi}^u_f} & = 0
		\end{split}
		\end{equation}
		%
		for all test-functions $\left( (\hat{\varphi}^u_s, \hat{\varphi}^u_f), (\hat{\varphi}^v_s, \hat{\varphi}^v_f), \hat{\varphi}^p_f \right)$ in $V_u \times V_v^0 \times V_p$,
		where $(\cdot ,\cdot)_\Omega$ denotes the $L^2(\Omega)$ inner product for
		vector-functions, and
		$\left\langle \cdot, \cdot \right\rangle_{\hat{\Gamma}_{out}}$ is nothing but
		the duality product of the trace of functions from $H^1(\hat{\Omega}_f)^{d}$
		on $\hat{\Gamma}_{out}$ with functionals from the dual space.
	\end{form}

	\begin{remark}[Well-posedness of the nonlinear fluid mesh motion problem]
		The nonlinear fluid mesh motion problem is a quasi-linear
		problem (for the definition of `quasi-linear' we refer to \cite{Evans2010}), which can be analyzed in the framework of monotone operators (chapter 9 of \cite{Evans2010}) as long as we can guarantee that $\hat J > 0$. Then, the solution $\hat{u}_f$ is indeed a $H^1$ function. We notice that the chosen function spaces for the fluid and solid subproblems are conforming with the available theory \cite{Temam:NavierStokes,Grandmont2002,Ball76,Ciarlet90}. Of course a complete well-posedness analysis is not available yet in the literature.
	\end{remark}

	\begin{remark}[Well-posedness of fluid-structure interaction]
		We assume in the following that there exists a unique smooth solution for the
		variational FSI problem. For more information on the assumptions imposed on initial
		data and regularity of the domain to guarantee existence and uniqueness for FSI,
		we refer for example to \cite{WellPosednessFSI,Grandemont2008,CoutandShkoller,CouShk06}.
	\end{remark}

	\subsection{Discretization in Time}

	For the time discretization, we employ a one-step theta scheme as given below.
	\begin{definition}[One-Step-Theta Scheme]
		Given a differential equation $a(u) \partial_t u + A(u) = 0$, the one-step-theta scheme reads as follows:
		$$
		\left(\theta \ a(u^n) + (1-\theta) \ a(u^{n-1}) \ (u(t^n) - u(t^{n-1})\right) - \Delta t \ \theta \ A(u^n) - \Delta t \ (1-\theta) \ A(u^{n-1}) = 0
		$$
		with $\theta \in [0,1]$ and $\Delta t := t^n - t^{n-1}$.
	\end{definition}

	Different values of $\theta$ result in time-stepping schemes with different properties. Popular choices are
	$\theta = 0$ (Explicit Euler) \cite{timestepping},
	$\theta = 0.5$ (Crank-Nicolson),
	$\theta = 0.5 + \Delta t$ (Shifted Crank-Nicolson) \cite{HeRa90},
	$\theta = 1$ (Implicit Euler),
	or the
	fractional-step-theta scheme \cite{BristeauGlowinskiPeriaux1987}.

	The explicit Euler method would require too small timesteps
	in order to be stable,
	whereas the Implicit Euler method dampens the oscillations too much.
	Hence, we focus on the Crank-Nicolson method and its shifted variant,
	since these methods provide the correct results.
	Detailed computational comparisons of these time-stepping
	schemes for fluid-structure interaction problems
	were performed in \cite{Wi11_phd,RiWi14_springer}. In \cite{Wi15_Hanoi} it has
	been shown that the choice of $\theta$ significantly influences the solid
	displacement. Choosing $\theta = 0.6$ (or even larger up to $\theta = 1$) does not yield any deformation
	of the elastic beam in the FSI-2 benchmark.
	\begin{remark}[Implicit Pressure]
		The pressure term within the Navier-Stokes equations is treated fully implicit,
		i.e., just as in the case $\theta = 1$,
		independent of the actual choice of $\theta$.
		This is motivated by the theory of differential algebraic equations, see e.g. \cite{Weickert:StokesDAE}.
	\end{remark}

	\subsection{Linearization}

	The nonlinearities within our FSI system are treated by Newton's method. Details are presented in \cite{Fernandez:Newton}.

	\begin{algorithm}
		Let $A(u)(\varphi)$ be a semi-linear form (linear with respect to the second argument), $F(\varphi)$ a linear form. Then the solution of $A(u)(\varphi) = F(\varphi)$ can be obtained by the following iteration:
		\begin{algorithmic}[1]
			\State Initial guess $u_0$
			\For{$k=0, 1, \dots$ until convergence}
			\State Solve $A'(u_k)(\varphi, \delta u_k) = F(\varphi) - A(u_k)(\varphi)$ for $\delta u_k$
			\State Update $u_{k+1} := u_k + \delta u_k$
			\EndFor
		\end{algorithmic}
		\caption{Newton Linearization}
		\label{alg:newton}
	\end{algorithm}

	For the differentiation of the whole FSI operator (\ref{fsi:all}), we need the derivatives of quantities like $\hat{J}, \hat{F}, \hat{F}^{-1}$ and others. Since we will use the notation $A'$ for the full Jacobian, we denote the following derivatives by
	$\partial_u \hat{F}(u)$ or simply $\partial \hat{F}$ etc., if it is clear by which variable we are differentiating.

	\begin{theorem}[FSI-related Derivatives]
		It holds:
		\begin{equation}
		\begin{split}
		\partial \hat{F} & = \hat{\nabla} \delta \hat{u} \\
		\partial \hat{J} & = \hat{J} \text{tr}(\hat{F}^{-1} \hat{\nabla} \delta \hat{u}) \\
		\partial \hat{F}^{-1} & = - \hat{F}^{-1} \hat{\nabla} \delta \hat{u} \hat{F}^{-1} \\
		\partial \hat{F}^{-T} & = (\partial \hat{F}^{-1})^T \\
		\partial \text{tr}(E) & = \text{tr}(\partial E)
		\end{split}
		\end{equation}
	\end{theorem}

	\begin{proof}
		Unless trivial, see for example \cite{KontMech, Elasticity}.
	\end{proof}

	\newpage
	The semi-linear form $A(U)(\Phi)$ and linear form $F(\Phi)$ are given as the sum of all equations in (\ref{eq:weak:FSI}).

	\begin{equation}
	\label{fsi:all}
	\begin{split}
	A(U)(\Phi) & = \spfhat{\hat{\rho}_f \hat{J}^\theta (\hat{v}_f - \hat{v}_f^{n-1}) }{\hat{\varphi}^v_f} - \spfhat{\hat{J} \hat{\rho}_f \hat{F}^{-1} \hat{\nabla} \hat{v}_f \cdot \hat{u}_f}{\hat{\varphi}^v_f}
	\\
	& +	\Delta t \ \theta \left[ \spfhat{\hat{J} \hat{\rho}_f \hat{F}^{-1} \hat{\nabla} \hat{v}_f \cdot \hat{v}_f}{\hat{\varphi}^v_f}
	+ \spfhat{\hat{J} \hat{\sigma}_f^v \hat{F}^{-T}}{\hat{\nabla} \hat{\varphi}^v_f} \right]
	\\
	& - \Delta t \ \theta \left\langle \hat{\rho}_f \hat{\nu}_f \hat{J} \hat{F}^{-T} \hat{\nabla} \hat{v}_f^T \hat{F}^{-T} \cdot \hat{n}_f \ , \hat{\varphi}^v_f \right\rangle_{\hat{\Gamma}_{out}}
	\\
	& + \spfhat{\hat{J} (-\hat{p}_f) \hat{F}^{-T}}{\hat{\nabla} \hat{\varphi}^v_f}
	\\
	& + \spshat{\hat{\rho}_s \hat{v}_s^n }{\hat{\varphi}^v_s}
	+ \Delta t \ \theta \spshat{\hat{F} \hat{\Sigma}_s}{\hat{\nabla} \hat{\varphi}^v_s}
	\\
	& +	\spshat{\hat{u}_s}{\hat{\varphi}^u_s} - \Delta t \ \theta \ \spshat{\hat{v}_s}{\hat{\varphi}^u_s}
	\\
	& + \spfhat{\dfrac{1}{\hat{J}} \hat{\nabla} \hat{u}_f}{\hat{\nabla} \hat{\varphi}^u_f}
	\end{split}
	\end{equation}
	and
	\begin{equation}
	\begin{split}
	F(\Phi) & = \spshat{\hat{\rho}_s \hat{v}_s^{n-1}}{\hat{\varphi}^v_s}
	- \Delta t \ (1 - \theta) \spshat{\hat{F}^{n-1} \hat{\Sigma}_s^{n-1}}{\hat{\nabla} \hat{\varphi}^v_s}
	\\
	& + \spshat{\hat{u}_s^{n-1}}{\hat{\varphi}^u_s} + \Delta t \ (1 - \theta) \ \spshat{\hat{v}_s^{n-1}}{\hat{\varphi}^u_s}
	\\
	& - \spfhat{\hat{J} \hat{\rho}_f \hat{F}^{-1} \hat{\nabla} \hat{v}_f \cdot \hat{u}_f^{n-1}}{\hat{\varphi}^v_f}
	\\
	& - \Delta t \ (1-\theta) \left[ \spfhat{\hat{J} \hat{\rho}_f \hat{F}^{-1} \hat{\nabla} \hat{v}_f \cdot \hat{v}_f}{\hat{\varphi}^v_f}
	+ \spfhat{\hat{J} \hat{\sigma}_f^v \hat{F}^{-T}}{\hat{\nabla} \hat{\varphi}^v_f} \right]^{n-1}
	\\
	& + \Delta t \ (1-\theta) \left[ \left\langle \hat{\rho}_f \hat{\nu}_f \hat{J} \hat{F}^{-T} \hat{\nabla} \hat{v}_f^T \hat{F}^{-T} \cdot \hat{n}_f \ , \hat{\varphi}^v_f \right\rangle_{\hat{\Gamma}_{out}} \right]^{n-1}
	\end{split}
	\end{equation}
	with the test-function $\Phi :=
	(\hat{\varphi}^v_f, \hat{\varphi}^u_f, \hat{\varphi}^p_f, \hat{\varphi}^v_s, \hat{\varphi}^u_s)$
	and the solution variable $U := (\hat{v}_f, \hat{u}_f, \hat{p}_f, \hat{v}_s, \hat{u}_s)$.

	Due to the nested non-linearities, which all require multiple
	applications of the product rule, the computation of all single terms
	of the Jacobian is quite lengthy. However, when implementing the
	derivatives, we do not actually need to expand all terms.
	Therefore, we do not write them down explicitly here for the
	convenience of the reader.

	\subsection{Spatial Discretization}
	Now we are in the position to discretize the Jacobian and Newton
	residual given in Algorithm \ref{alg:newton}.
	We are going to use a quadrilateral ($2d$) or hexahedral ($3d$)
	subdivision
	$ \mathcal{T}_h $ of the reference domain
	${\overline{\hat{\Omega}}} = \bigcup_{T \in \mathcal{T}_h} \overline{T}$
	with $Q(k)$ shape functions for displacements and velocities, and discontinuous $P(k-1)$ elements for the pressure, where $h$ denotes the usual discretization parameter.
	The subdivision matches the interface, i.e., every element $T$ is either part of the fluid or the solid domain, but not both.

	As indicated above we use the following discrete function spaces
	\begin{equation}
	\begin{split}
	V_h^s & := \left\{ v \in H^1(\hat{\Omega}_s)^d : v \bigr\rvert_{T} \in Q(k) \hspace{3mm} \forall\; T \subset \hat{\Omega}_s \right\}, \\
	V_h^f & := \left\{ v \in H^1(\hat{\Omega}_f)^d : v \bigr\rvert_{{T}} \in Q(k) \hspace{3mm} \forall\; T \subset \hat{\Omega}_f \right\}, \\
	L_h^f & := \left\{ v \in L^2(\hat{\Omega}_f) : v \bigr\rvert_{{T}} \in P(k-1) \hspace{1mm} \forall \;  T \subset  \hat{\Omega}_f \right\}
	\end{split}
	\end{equation}
	with the nodal basis functions
	\begin{equation}
	\begin{split}
	V_h^s & = span\{\varphi_{s,h}^v[j], \ j = 1,\dots, N_{v_s} \} = span\{\varphi_{s,h}^u[j], \ j = 1,\dots, N_{u_s} \}, \\
	V_h^f & = span\{\varphi_{f,h}^v[j], \ j = 1,\dots, N_{v_f} \} = span\{\varphi_{f,h}^u[j], \ j = 1,\dots, N_{u_f} \}, \\
	L_h^f & = span\{ \varphi_{f,h}^p[j], \ j = 1,\dots, N_{p_f} \}.
	\end{split}
	\end{equation}

	Using the ansatz $$\delta U := \sum_{j = 1}^{N_{FSI}} \delta U_j \Phi_j$$ for the Newton correction at each step $k$ leads to the linear system $A_h \delta U_h = F_h$, with the system matrix
	\begin{equation}
	\label{mat:small}
	A_h :=
	\begin{bmatrix}
	\mathcal{M} & \mathcal{C}_{ms} & 0 \\
	\mathcal{C}_{sm} & \mathcal{S} & \mathcal{C}_{sf} \\
	\mathcal{C}_{fm} & \mathcal{C}_{fs} & \mathcal{F}
	\end{bmatrix},
	\end{equation}
	where $\mathcal{M}, \mathcal{S}$ and $\mathcal{F}$ denote the discrete versions of the mesh-motion, solid and fluid equations, respectively. The coupling terms $\mathcal{C}_{**}$ arise because of the ALE transformation and interface coupling conditions.

	\section{Approximate Block-LDU - Preconditioner}
	\label{sec:pc}

	\subsection{Approximate factorization}

	By treating our block system as a simple $3 \times 3$ matrix, we can (formally) apply an LDU-factorization, yielding the following decomposition:

	\begin{equation}
	\begin{bmatrix}
	I & 0 & 0 \\
	\frac{C_{sm}}{M} & I & 0 \\
	\frac{C_{fm}}{M} & \frac{C_{fs}-\frac{C_{fm}}{M}}{\mathcal{S} - \frac{\mathcal{C}_{sm} \mathcal{C}_{ms}}{\mathcal{M}}} & 1 \\
	\end{bmatrix}
	\begin{bmatrix}
	\mathcal{M} & 0 & 0 \\
	0 & \mathcal{S} - \frac{\mathcal{C}_{sm} \mathcal{C}_{ms}}{\mathcal{M}} & 0 \\
	0 & 0 & \mathcal{F}-\frac{\mathcal{C}_{sf} \left(\mathcal{C}_{fs}-\frac{\mathcal{C}_{fm} \mathcal{C}_{ms}}{\mathcal{M}}\right)}{\mathcal{S} - \frac{\mathcal{C}_{sm} \mathcal{C}_{ms}}{\mathcal{M}}} \\
	\end{bmatrix}
	\begin{bmatrix}
	I & \frac{\mathcal{C}_{ms}}{\mathcal{M}} & 0 \\
	0 & I & \frac{\mathcal{C}_{sf}}{\mathcal{S}- \frac{\mathcal{C}_{sm}\mathcal{C}_{ms}}{\mathcal{M}}} \\
	0 & 0 & I \\
	\end{bmatrix}
	\end{equation}
	making slight abuse of the notation by using $\frac{A}{B}$ as shortcut for $AB^{-1}$ or $B^{-1}A$.
	Using this LDU-factorization directly is obviously not
	very efficient since it involves the computation of too many inverses
	(displayed as fractions) and matrix-matrix products.

	In order to simplify the computation, we drop the term $\mathcal{C}_{sm}$ from our system. This is justified by the fact that this coupling term corresponds to the term which is set to zero on the interface to avoid the (non-physical) coupling from the mesh into the solid equations.

	The simplified LDU decomposition is then given as

	\begin{equation}
	\begin{bmatrix}
	I & 0 & 0 \\
	0 & I & 0 \\
	\frac{\mathcal{C}_{fm}}{\mathcal{M}} & \frac{\tilde{\mathcal{C}}_{fs}}{\mathcal{S}} & 1 \\
	\end{bmatrix}
	\begin{bmatrix}
	\mathcal{M} & 0 & 0 \\
	0 & \mathcal{S} & 0 \\
	0 & 0 & \mathcal{X} \\
	\end{bmatrix}
	\begin{bmatrix}
	I & \frac{\mathcal{C}_{ms}}{\mathcal{M}} & 0 \\
	0 & I & \frac{\mathcal{C}_{sf}}{\mathcal{S}} \\
	0 & 0 & I \\
	\end{bmatrix}
	\end{equation}

	with $\tilde{\mathcal{C}}_{fs} = \mathcal{C}_{fs}-\frac{\mathcal{C}_{fm} \mathcal{C}_{ms}}{\mathcal{M}}$ and $\mathcal{X} = \mathcal{F}-\frac{\mathcal{C}_{fs} \tilde{\mathcal{C}}_{sf}}{\mathcal{S}}$.

	In order to apply solvers like sparse-LU or AMG methods, which require explict knowledge about the matrix entries, we would have to compute $\mathcal{X}$ explicitly. However, this involves the explicit computation of $\mathcal{S}^{-1}$, which we want to avoid. One possibility is to replace $\mathcal{S}$ by a block-diagonal approximation that can be easily inverted, as shown in \cite{LaYa2016}. In our application, we simply ignore the perturbation $\frac{\mathcal{C}_{sf} \tilde{\mathcal{C}}_{fs}}{\mathcal{S}}$, and observe almost no difference compared to the block-diagonal approximation. This factorization can now be applied to our block-system as given in Algorithm \ref{alg:LY}.

	\begin{remark}
		The LDU-factorization depends on the ordering of the blocks and is therefore not unique, giving rise to many different preconditioners. For a comparison of some of them as well as further numerical results we refer to our previous work
		\cite{JodlbauerWick:2017}.
	\end{remark}

	\begin{algorithm}
		\begin{algorithmic}[1]
			\State Solve $x_m = \mathcal{M}^{-1} r_m$
			\State Solve $x_s = \mathcal{S}^{-1} r_s$
			\State Solve $x_f = \mathcal{F}^{-1} (r_f - \mathcal{C}_{fm} x_m - \mathcal{C}_{fs} x_s)$
			\State Update $x_s = x_s - \mathcal{S}^{-1} \mathcal{C}_{sf} x_f$
			\State Update $x_m = x_m - \mathcal{M}^{-1} \mathcal{C}_{ms} x_s$
		\end{algorithmic}
		\caption{Evaluation of $P^{-1} r$.}
		\label{alg:LY}
	\end{algorithm}

	The approximate solution of the subproblems, i.e.,
	application of the inverses in Algorithm~\ref{alg:LY},
	is discussed in the next sections.

	\subsection{Solving the Mesh Subproblem}

	In our configuration, the mesh-motion equation is a scaled Laplace-type equation. Thus, available AMG methods like those provided by
	the Trilinos package \cite{Trilinos, Trilinos:ML} are good
	candidates for the approximation of $\mathcal{M}^{-1}$.

	\subsection{Solving the Solid Subproblem}

	For the solution of the solid equations, we employ a Schur-complement approach, which eliminates the equations related to the solid velocity $\hat{v}_s$.
	Here, matrices $M_{**}$ denote mass-matrices, and $K$ denotes the matrix resulting from $\spshat{\hat{F} \hat{\Sigma}_s}{\hat{\nabla} \hat{\varphi}^v_s}$. The solution of
	the linear system
	\begin{equation*}
	\begin{bmatrix}
	\hat\rho_s M_{vv} & \Delta t \ \theta K_{vu} \\
	-\Delta t \ \theta M_{uv} & M_{uu}
	\end{bmatrix}
	\begin{bmatrix}
	x_{v_s} \\ x_{u_s}
	\end{bmatrix}
	=
	\begin{bmatrix}
	r_{v_s} \\ r_{u_s}
	\end{bmatrix}
	\end{equation*}
	is equivalent to solving
	\begin{equation*}
	\begin{bmatrix}
	\hat\rho_s M_{vv} + \Delta t^2 \ \theta^2 K_{vu} M_{uu}^{-1} M_{uv} & 0 \\
	-\Delta t \ \theta M_{uv} & M_{uu}
	\end{bmatrix}
	\begin{bmatrix}
	x_{v_s} \\ x_{u_s}
	\end{bmatrix}
	=
	\begin{bmatrix}
	r_{v_s} - \Delta t \ \theta K_{vu} M_{uu}^{-1} r_{u_s} \\ r_{u_s}
	\end{bmatrix}
	\end{equation*}

	Since we use equal-order elements for the displacement and velocity variables,
	the mass matrices $M_{**}$ are all equal (possibly after reordering the dofs).
	Hence, we can simplify the system above to
	\begin{equation}
	\begin{bmatrix}
	\hat\rho_s M + \Delta t^2 \ \theta^2 K_{vu} & 0 \\
	-\Delta t \ \theta \ M & M
	\end{bmatrix}
	\begin{bmatrix}
	x_{v_s} \\ x_{u_s}
	\end{bmatrix}
	=
	\begin{bmatrix}
	r_{v_s} - \Delta t \ \theta K_{vu} M^{-1} r_{u_s} \\ r_{u_s}
	\end{bmatrix}
	\end{equation}
	This system is then solved in a similar fashion as the global system:

	\begin{algorithm}
		\begin{algorithmic}[1]
			\State Solve $x_{v_s} = (\hat\rho_s M + \Delta t^2 \ \theta^2 K_{vu})^{-1} (r_{v_s} - \Delta t \ \theta K_{vu} M^{-1} r_{u_s})$
			\State Solve $x_{u_s} = M^{-1} (r_{u_s} + \Delta t \ \theta M x_{v_s})$
		\end{algorithmic}
		\caption{Evaluation of $P^{-1} r$.}
	\end{algorithm}

	\noindent
	We do not require the exact realization of the application of the occurring inverses, but again use an AMG-solver to obtain reasonable approximations to $(\hat\rho_s M + \Delta t^2 \ \theta^2 K_{vu})^{-1}$ and $M^{-1}$.

	\subsection{Solving the Fluid Subproblem}

	For the fluid inverse $\mathcal{F}^{-1}$, we employ the same Schur-complement
	approach as for the solid system (an Uzawa-like method). Therein, we use an
	AMG-preconditioned GMRES solver to approximate the action of the occurring
	inverses. The additional solver was necessary to cope with
	large solid deformations,
	because
	a standalone ML-AMG (multi-level algebraic multigrid)
	method did no longer suffice.

	\section{Parallel Implementation}
	\label{sec:parallel}

	The implementation uses the C++ library
	deal.II \cite{BangerthHartmannKanschat2007,dealII84}, with the
	Trilinos package \cite{Trilinos} for linear algebra operations and its
	multi-level package ML \cite{Trilinos:ML}. Partitioning of the mesh is
	done using ParMETIS \cite{KarKu99}.

	\subsection{Mesh Partitioning}

	In a distributed setting, each core only stores parts of the problem. Hence, the mesh is split into several subdomains using ParMETIS. In our tests, we considered the following splitting strategies:

	\begin{itemize}
		\item shared: each core owns parts of both, the fluid and the solid domain
		\item split: each core owns either parts of the fluid or the solid domain, but not both
		\item default: no distinction is made between fluid and solid subdomains
	\end{itemize}

	\begin{figure}[ht]
		\centering
		\includegraphics[width=0.85\textwidth]{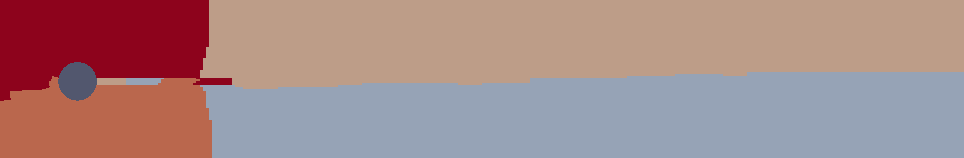}
		\caption{Shared configuration using four subdomains.}
	\end{figure}

	\begin{figure}[ht]
		\centering
		\includegraphics[width=0.85\textwidth]{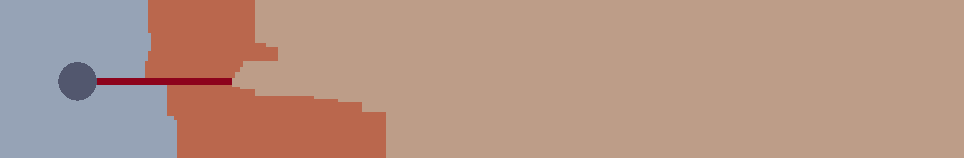}
		\caption{Split configuration using four subdomains. Load balancing issues are obvious.}
	\end{figure}

	Obviously, the split-strategy suffers from load balancing issues if a small number of cores is used, since it is not possible to obtain an almost uniform distribution of the dofs. This kind of problem is avoided using the shared-type partitioning. On the other hand, the split-strategy may require less communication than the shared-strategy due to the separation of fluid and solid. In addition to the owned cells, each core also obtains information about the neighboring cells, which are referred to as ghost cells or ghost layer.

	\subsection{Interface coupling}

	We use global variables for the displacement, velocity and pressure. The benefit of this approach is
	the fact
	that all interface conditions, namely the continuity of shape-functions, velocities and displacements, are automatically fulfilled. The natural interface condition \eqref{eq:fsi:strong:normal:stresses}, the continuity of normal stresses, is fulfilled in a weak sense.

	\subsection{Distribution of dofs}

	Degrees of freedoms are distributed in a similar manner as the mesh. Each
	CPU
	owns the dofs located in the interior of its associated subdomain. Dofs located on the interface between different cores are assigned to one of the adjacent ones.
	Hence, each dof is owned by exactly one core, although possibly more cores may acquire information about this dof. The set of all dofs that are required by a specific core is called the set of locally relevant dofs within the deal.II library.

	\section{Numerical Results}
	\label{sec:results}

	In this section, we consider two test cases to demonstrate the computational performance of
	the solvers presented in Section~\ref{sec:pc}.
	In the first example, we focus on the FSI-2 benchmark \cite{TuHrBenchmark}, which exhibits large solid deformations and is a well-known difficult test problem.
	The second example, which describes the flow around an obstacle,
	is a three-dimensional configuration inspired by \cite{Richter:GMG}. In both test cases, we first focus on the correct physics and reproduce the quantities of interest published in the literature.
	Then, we discuss the parallel performance of the preconditioner from  Section~\ref{sec:pc}.
	We notice that numerical results regarding robustness, $h$ and $\Delta t$-dependence
	have been presented in \cite{Jodl16,JodlbauerWick:2017}
	for sequential computations.
	Consequently, in this work, we focus on the parallel scalability.

	All tests regarding parallelization were done on the distributed memory cluster
	Radon1\footnote{https://www.ricam.oeaw.ac.at/hpc/overview/}
	at RICAM, Linz.
	Radon1 consists of 64 compute nodes each with two 8-core Intel Haswell processors
	(Xeon E5-2630v3, 2.4Ghz) and 128 GB of memory.

	In all our tests, all partitioning strategies (shared, split, default) yielded similar scalability results. However, the shared-type partition required significantly less time than the others. Unless stated otherwise, all figures below are done using the faster shared-type strategy.

	For solving the nonlinear problem we employ a Quasi-Newton scheme, which reassembles the Jacobian only if the reduction of the last Newton iteration is less than $10$, i.e. assemble if
	$\Vert r_k \Vert_\infty > 0.1\, \Vert r_{k-1} \Vert_\infty$.

	The Newton solver is stopped once the nonlinear residual $r$ satisfies
	$\Vert r_k \Vert_\infty < 10^{-6}\Vert r_0 \Vert_\infty$, where $r_0$ is the
	initial Newton residual.

	GMRES iterates until a reduction of $10^3$ is achieved, which yields only a slight increase in the same number of Newton iterations compared to solving the linear systems with a direct method.

	\subsection{Example 1: FSI-2 benchmark}

	\subsubsection{Description}

	The geometry consists of a channel with some given inflow velocity
	profile on the left boundary. Inside the channel is a fixed cylindric
	obstacle with an elastic beam attached to it, as depicted in
	Figure \ref{fsi:geometry}.

	\begin{figure}[!ht]
		\centering
		\begin{tikzpicture}
		\begin{scope}[scale=5.0]
		\coordinate (p0) at (0, 0);
		\coordinate (p1) at (2.5, 0.0);
		\coordinate (p2) at (2.5, 0.41);
		\coordinate (p3) at (0.0, 0.41);

		\coordinate (a1) at (0.0, 0.068);
		\coordinate (a2) at (0.0, 0.136);
		\coordinate (a3) at (0.0, 0.205);
		\coordinate (a4) at (0.0, 0.273);
		\coordinate (a5) at (0.0, 0.341);

		\coordinate (b1) at (0.05, 0.068);
		\coordinate (b2) at (0.1, 0.136);
		\coordinate (b3) at (0.12, 0.205);
		\coordinate (b4) at (0.1, 0.273);
		\coordinate (b5) at (0.05, 0.341);

		\coordinate (f0) at (0.6, 0.19);
		\coordinate (f1) at (0.6, 0.21);
		\coordinate (f2) at (0.24899, 0.21);
		\coordinate (f3) at (0.24899, 0.19);

		\coordinate (C) at (0.2, 0.2);

		\draw (p0) -- (p1) node[midway, below] {$\hat\Gamma_{bottom}$};
		\draw (p1) -- (p2) node[midway, right] {$\hat\Gamma_{out}$};
		\draw (p2) -- (p3) node[midway, above] {$\hat\Gamma_{top}$};
		\draw (p3) -- (p0) node[midway, left] {$\hat\Gamma_{in}$};

		\draw[->] (a1) -- (b1);
		\draw[->] (a2) -- (b2);
		\draw[->] (a3) -- (b3);
		\draw[->] (a4) -- (b4);
		\draw[->] (a5) -- (b5);

		\draw (C) circle (0.05);

		\draw (f2) -- (f1) -- (f0) -- (f3);
		\end{scope}
		\end{tikzpicture}

		\caption{Geometry of the FSI-2 benchmark \cite{TuHrBenchmark}:
			Elastic beam immersed in a flow around a cylinder.}
		\label{fsi:geometry}
	\end{figure}
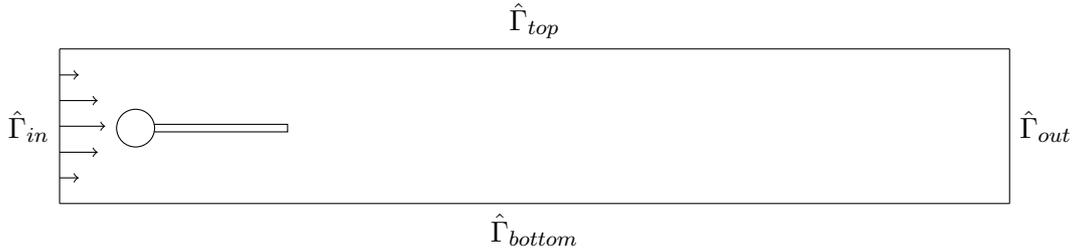

	The inflow is given by $ \hat{v}_f(t, (0, y)) = 6 \frac{y (H - y)}{H^2} s(t) \bar{v}$,
	where the height of the channel $H$ is given by $H = 0.41$,
	and $s(t) = \frac{1}{2} (1 - \cos(\frac{\pi}{2} t))$  is nothing but
	a time-dependent smoothing factor, and $\bar{v} = 1$ is the mean inflow velocity.
	The other quantities of the geometry are given as follows:\\

	\begin{minipage}{.3\textwidth}
		\begin{tabular}{|c|c|}
			\hline
			Quantity & Value \\
			\hline
			channel length & $2.5$ \\
			channel height & $0.41$ \\
			cylinder center & $(0.2, 0.2)$ \\
			cylinder radius & $0.05$ \\
			beam thickness & $0.02$ \\
			beam length & $0.35$ \\
			reference point & $(0.6, 0.2)$ \\
			\hline
		\end{tabular}
	\end{minipage}
	\begin{minipage}{.3\textwidth}
		\begin{tabular}{|c|c|}
			\hline
			Quantity & Value \\
			\hline
			inflow velocity $\bar{v}$ & $1.0$ \\
			density & $10^{3}$ \\
			viscosity & $10^{-3}$ \\
			\hline
		\end{tabular}
	\end{minipage}
	\begin{minipage}{.3\textwidth}
		\begin{tabular}{|c|c|}
			\hline
			Quantity & Value \\
			\hline
			solid density $\hat\rho_s$ & $10^4$ \\
			Lam\'e $\lambda$ & $ 2 \cdot 10^6$ \\
			Lam\'e $\mu$ & $0.5 \cdot 10^{6}$ \\
			Poisson ratio & $0.4$ \\
			\hline
		\end{tabular}
	\end{minipage}

	\vspace*{2ex}
	\noindent
	We refer the reader to FeatFlow\footnote{http://www.featflow.de/en/benchmarks.html} or \cite{TuHrBenchmark}
	for a more detailed description of FSI-2 benchmark.

	\begin{figure}[H]
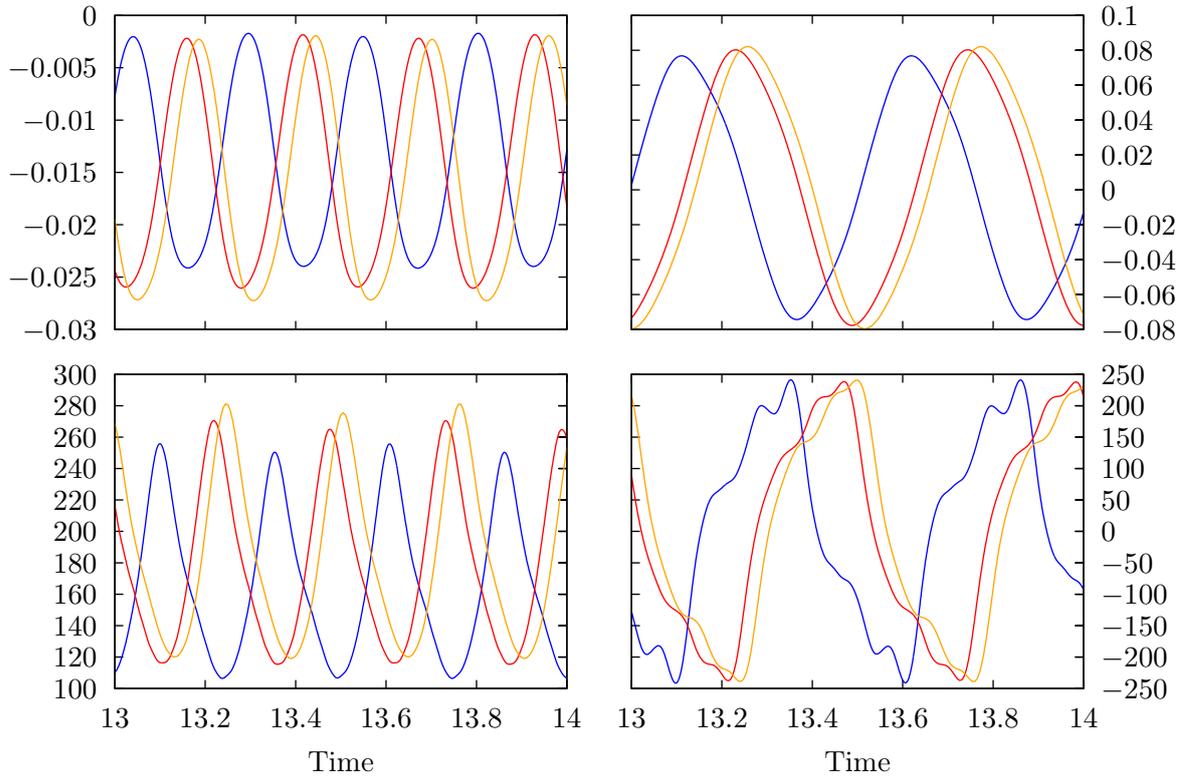

		\center
		\include{fsi2_q1_quantities_inc}
		\caption{FSI-2 benchmark: Drag, lift (bottom row) and
			x/y-displacements (top row)  using different levels $r$ of refinements for $Q(1)-Q(1)-P(0)$  (blue: $r = 2$, red: $r = 3$, orange: $r = 4$).}
		\label{fsi2:q1:quantities}
	\end{figure}

	\newpage
	\subsubsection{Computing and comparison of displacements, drag and lift}
	\label{sec_ex_1_goal_func}

	To check the correctness of our code, we first compare the evolution of
	the deflection of the tip of the elastic beam, drag and lift with the
	results published in the literature \cite{TuHrBenchmark,BuSc06,TW:MMPDE}.

	The displacement is measured at the reference point $(0.6, 0.2)$, located at the very end of the beam. Drag $F_D$ and lift $F_L$ are computed according to the following formula:
	\begin{equation}
	(F_D, F_L) = \int_{S} \hat\sigma_f \hat{n}_f ds,
	\end{equation}
	where $S$ denotes the boundary of flag and circle with the outer unit normal vector $\hat{n}_f$ pointing inside the solid domain.
	The results of these computations are shown in Figure \ref{fsi2:q1:quantities}, using a $Q(1)-Q(1)-P(0)$ discretization, and in Figure \ref{fsi2:q2:quantities} for $Q(2)-Q(2)-P(1)$.
	All configurations yield similar results, but a slight shift in time can be observed.

	\begin{figure}[H]
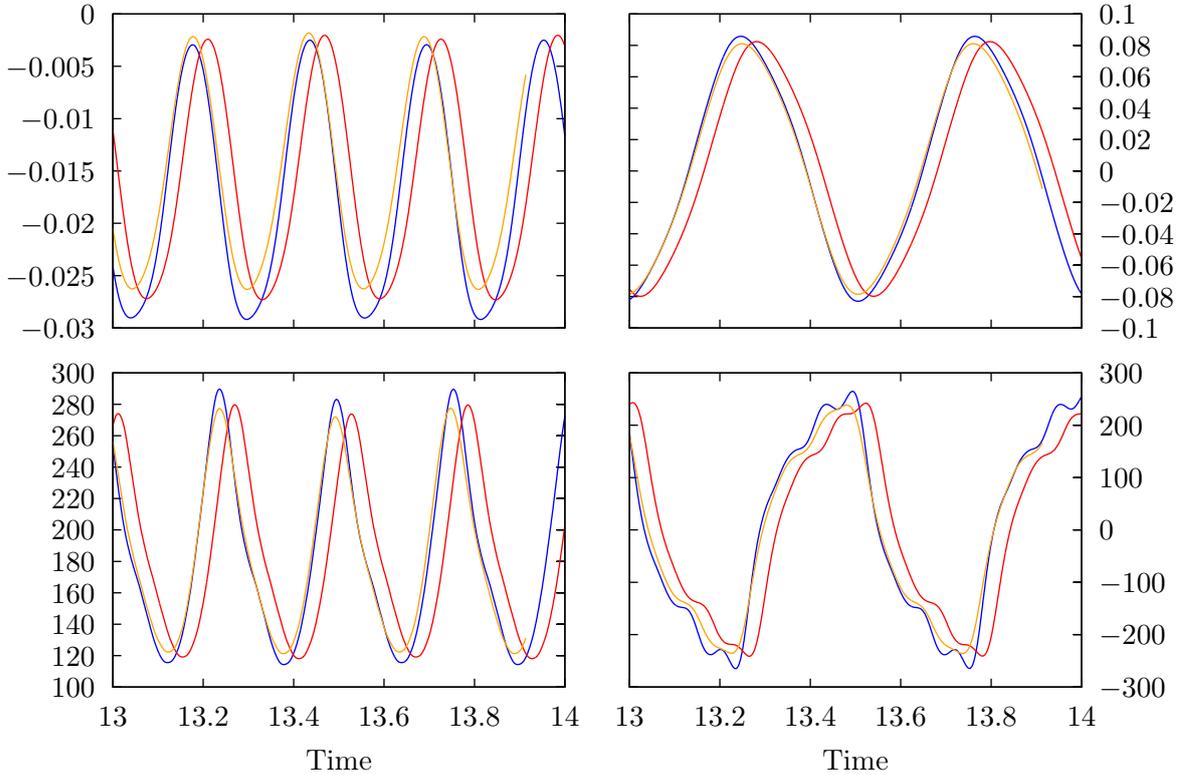

		\center
		\include{fsi2_q2_quantities_inc}
		\caption{FSI-2 benchmark: Drag, lift (bottom row) and x/y-displacements (top row) using different levels  $r$ of refinements for $Q(2)-Q(2)-P(1)$  (blue: $r = 1$, red: $r = 2$, orange: $r = 3$).}
		\label{fsi2:q2:quantities}
	\end{figure}

	\subsubsection{Parallel performance studies}

	\begin{figure}[H]
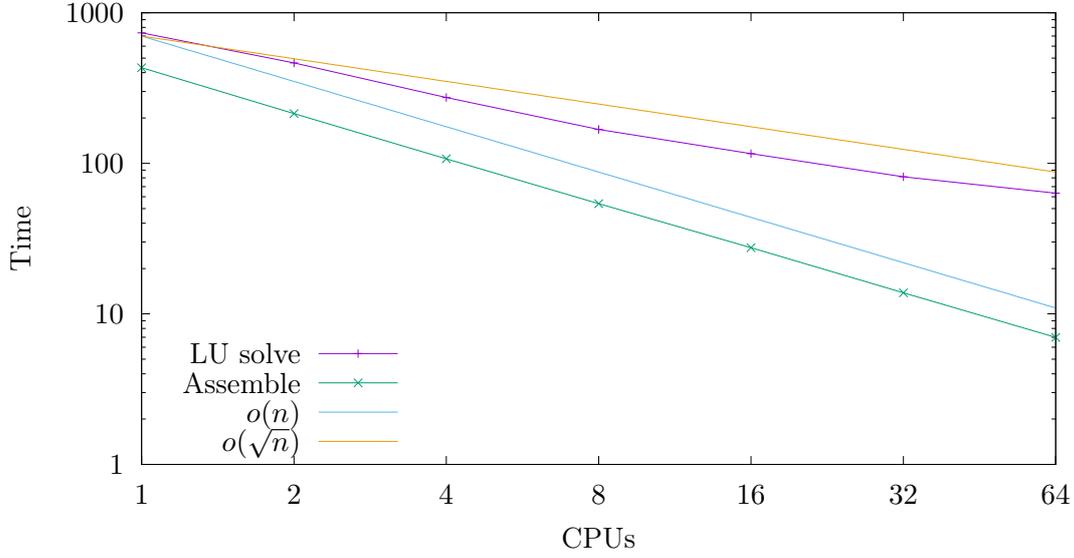

		\center
		\include{scala_direct_inc}
		\caption{FSI-2 benchmark: Strong scalability of the sparse direct solver MUMPS.}
		\label{fig:scala:direct}
	\end{figure}

	\begin{figure}[ht]
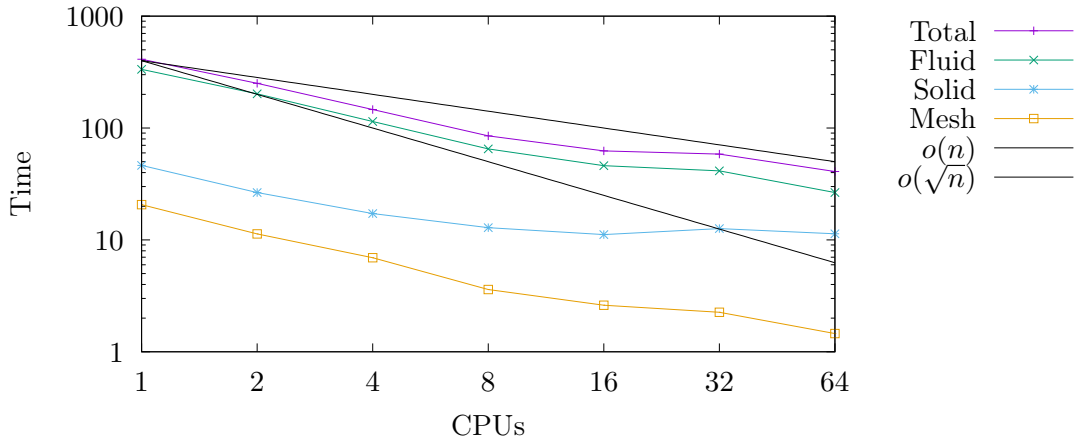

		\centering
		\include{scala_gmres_2d_inc}
		\caption{FSI-2 benchmark: Strong scalability using the
			preconditioned GMRES scheme for approximately $16 \cdot 10^6$
			dofs. Average time given in seconds for the solution of one linear system.}
		\label{fig:scala:gmres}
	\end{figure}

	For comparison, Figure \ref{fig:scala:direct} shows the parallel performance of the sparse direct solver MUMPS solving one monolithic linear system from the FSI-2 simulation. For a small number of cores, the speedup is somewhere between $\mathcal{O}(n)$ and $\mathcal{O}(\sqrt{n})$, but decays when more CPUs are added. Furthermore, it displays the expected perfect scalability of the assembling procedure, rendering the linear solver as the only remaining bottleneck in serial and parallel computations.

	Our preconditioned GMRES yields $\mathcal{O}(\sqrt{n})$ scalability as shown in Figure~\ref{fig:scala:gmres}.  This behavior is in agreement with the results presented in \cite{ParallelFSI}. We note that both methods do not yield the optimal scalability of $\mathcal{O}(n)$.

	The  time required
	for solving the linear system is dominated by the solution time  spent for
	the fluid sub-problem. The mesh problem is the easiest one to solve, requiring just a few AMG-cycles, hence contributing only little to the overall runtime. The solid problem does not scale very well, most likely due to its small size in comparison to the other problems.

	Specifically, using
	$1$ core, the total CPU time to solve the problem with $16\cdot 10^6$ dofs at a single time step
	is $411$ seconds, i.e., approximately $6.8$ minutes. On $4$ cores, the
	computational cost decreases to $146$ seconds, i.e., $2.4$ minutes. Thus we
	achieve a reduction by $64\%$ of the computational time. The further
	decrease using $64$ cores is less significant and drops to $41$ seconds
	per solution of the linear system.

	\subsection{Example 2: Flow around an elastic obstacle}

	\subsubsection{Description}
	This numerical test features a $3d$ flow around an elastic obstacle and is
	motivated from \cite{Richter:GMG}. The computational domain is given by $(0, L) \times (0, H) \times (-H, H)$, with the solid inclusion $(0.4, 0.5) \times (0, h) \times (-0.2, 0.2)$. Similar to the previous test, an inflow velocity is prescribed on the $yz$-plane by $\hat{v}_f(t, (0, y, z)) = \frac{81}{16} \frac{y (H - y) (H^2 - z^2)}{H^4} s(t) \bar{v}$.
	The geometrical parameters and $\bar{v}$ are given in Table~\ref{table:3d},
	whereas the
	fluid and solid parameters are chosen the same as in the FSI-2 benchmark.
	The geometry is illustrated in Figure~\ref{fig:3d}.

	\begin{center}
		\begin{tabular}{|c|c|}
			\hline
			Quantity & Value \\
			\hline
			channel length (x-direction) $L$ & $1.5$ \\
			channel height (y-direction) $H$ & $0.4$ \\
			channel width (z-direction) $2H$ & $0.8$ \\
			obstacle height $h$ & $0.3$ \\
			obstacle width & $0.4$ \\
			obstacle thickness & $0.1$ \\
			inflow velocity $\bar{v}$ & $3.0$ \\
			\hline
		\end{tabular}
		\captionof{table}{Geometry and problem data for Example 2.}
		\label{table:3d}
	\end{center}

	\begin{figure}[H]
		\center
		\includegraphics[width=0.8\textwidth]{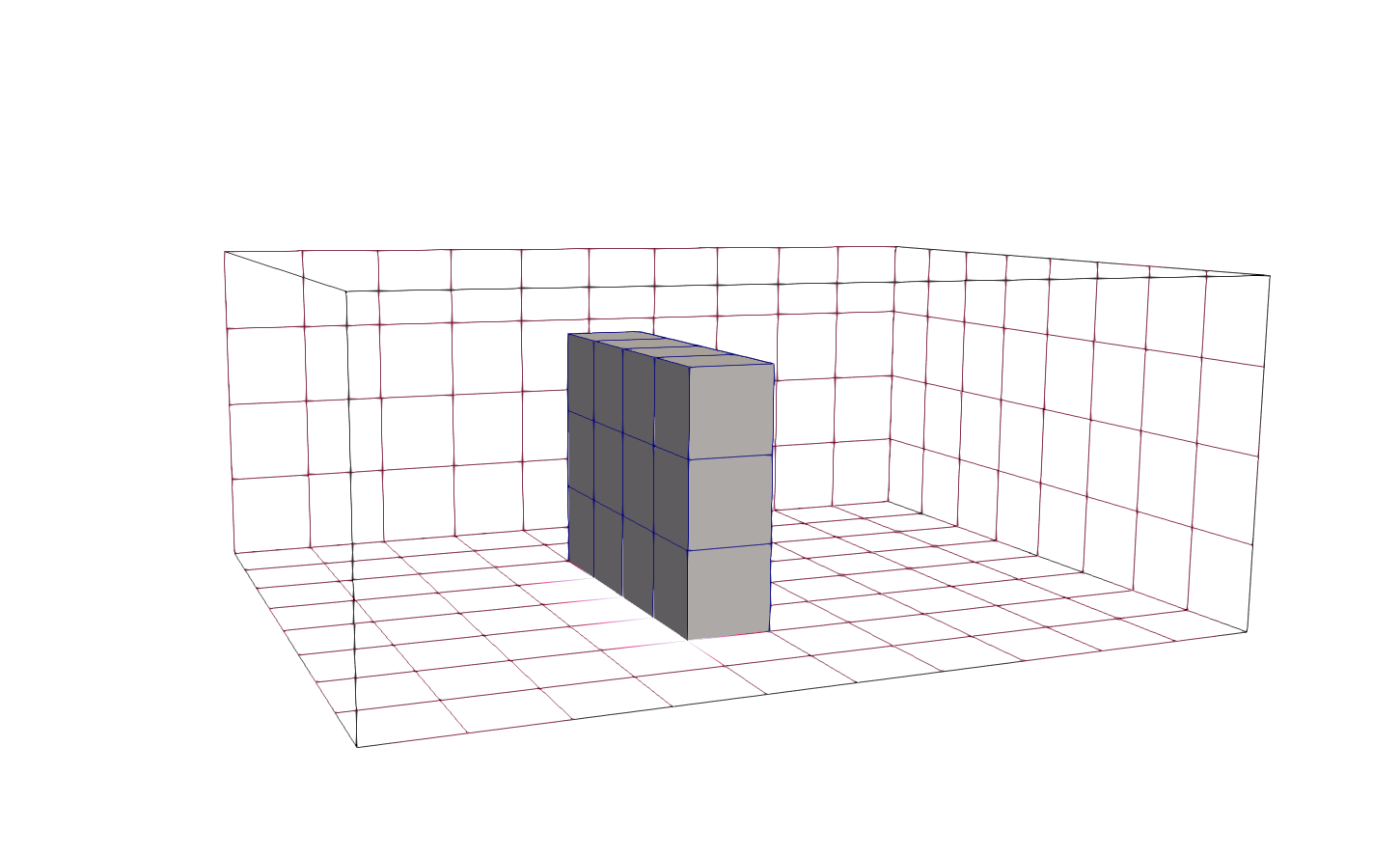}
		\caption{Graphical illustration of the geometry of Example
			2. The elastic obstacle is displayed in dark color. The
			explanation of the axes is given in Table \ref{table:3d}.}
		\label{fig:3d}
	\end{figure}

	\newpage
	\subsubsection{Evaluation of quantities of interest}

	As in Section \ref{sec_ex_1_goal_func}, we first compute the physical quantities of interest. This includes point evaluations at the upper boundary surface ($y = h$) of the solid obstacle at points given below
	The results are presented component-wise in
	Figures~\ref{fsi:richter:q1:quantities:x} -- \ref{fsi:richter:q1:quantities:z}
	using the scheme depicted in Table~\ref{table:points}.

	\begin{center}
		\begin{tabular}{|c|c|}
			\hline
			$P_{1} = (0.4, h, \phantom{-}0.0)$ & $P_{2}= (0.4, h, -0.2)$ \\
			\hline
			$P_{3} = (0.5, h, -0.2)$ & $P_{4}= (0.5, h, \phantom{-}0.0)$ \\
			\hline
		\end{tabular}
		\captionof{table}{Example 2: Evaluation points for the
			displacement of the elastic obstacle.}
		\label{table:points}
	\end{center}

	\begin{figure}
		\center
		\include{fsi_richter_q1_quantities_x_inc}
		\caption{Example 2: $x$-component of displacement $u(P_i)$ using different levels $r$ of refinements (blue: $r = 1$, red $r = 2$).}
		\label{fsi:richter:q1:quantities:x}
	\end{figure}

	\begin{figure}[H]
		\center
		\include{fsi_richter_q1_quantities_y_inc}
		\caption{Example 2: $y$-component of displacement $u(P_i)$ using different levels $r$ of refinements (blue: $r = 1$, red $r = 2$).}
		\label{fsi:richter:q1:quantities:y}
	\end{figure}

	\begin{figure}
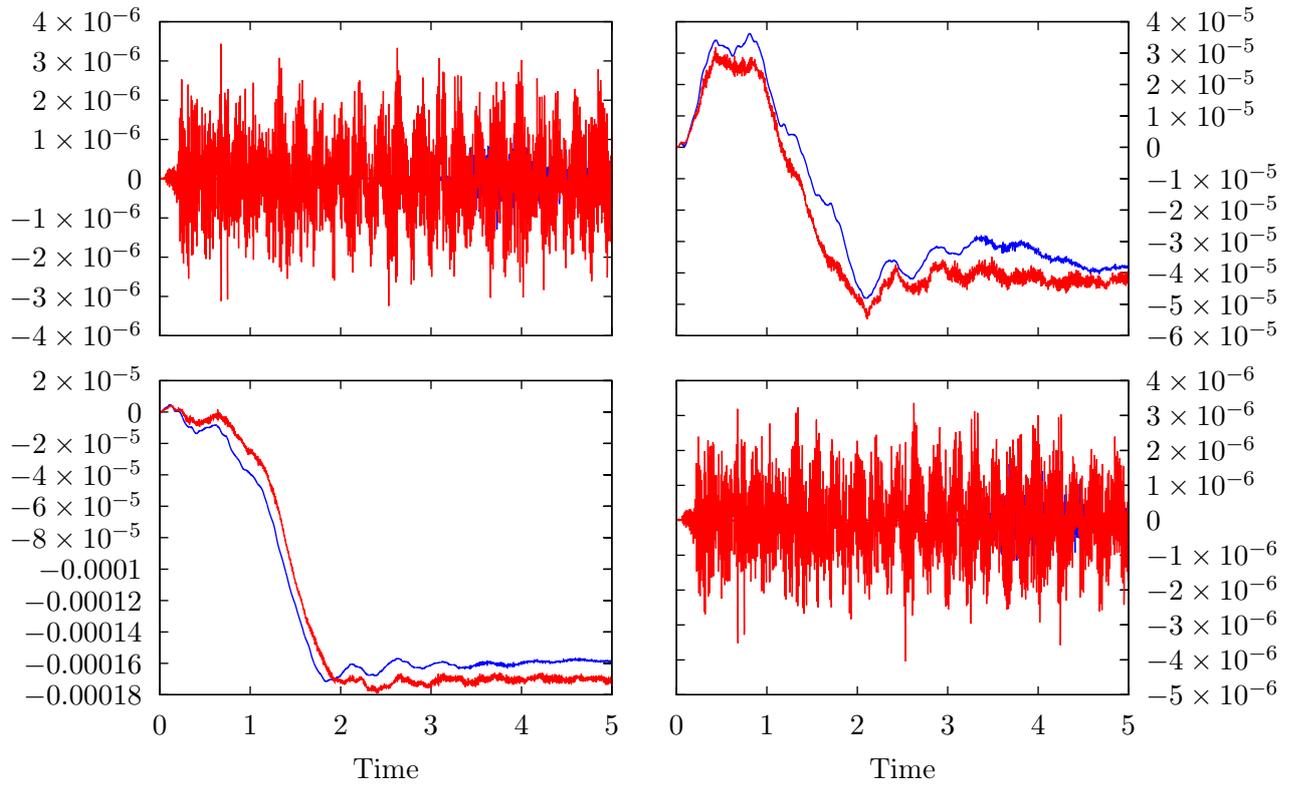

		\center
		\include{fsi_richter_q1_quantities_z_inc}
		\caption{Example 2: $z$-component of displacement $u(P_i)$
			using different levels $r$ of refinements (blue: $r = 1$, red
			$r = 2$). We notice that the oscillations are of order
			$10^{-6}$ and thus numerical noise. Thus, the
			$z$-displacements
			are approximately zero, which was expected due to the
			symmetry of the configuration.}
		\label{fsi:richter:q1:quantities:z}
	\end{figure}

	\newpage

	\subsubsection{Performance studies}

	We first start with comments on the nonlinear and linear solvers followed
	by an analysis of the parallel performance.
	Figure \ref{fig:newton_iter} shows that the number of Newton iterations remains approximately constant at $5-10$ iterations throughout the computations. We note that in $3d$, the number of iterations increases slightly during $h$-refinement. The respective results of the 2d benchmarks are discussed in \cite{JodlbauerWick:2017}.

	\begin{figure}[h!]
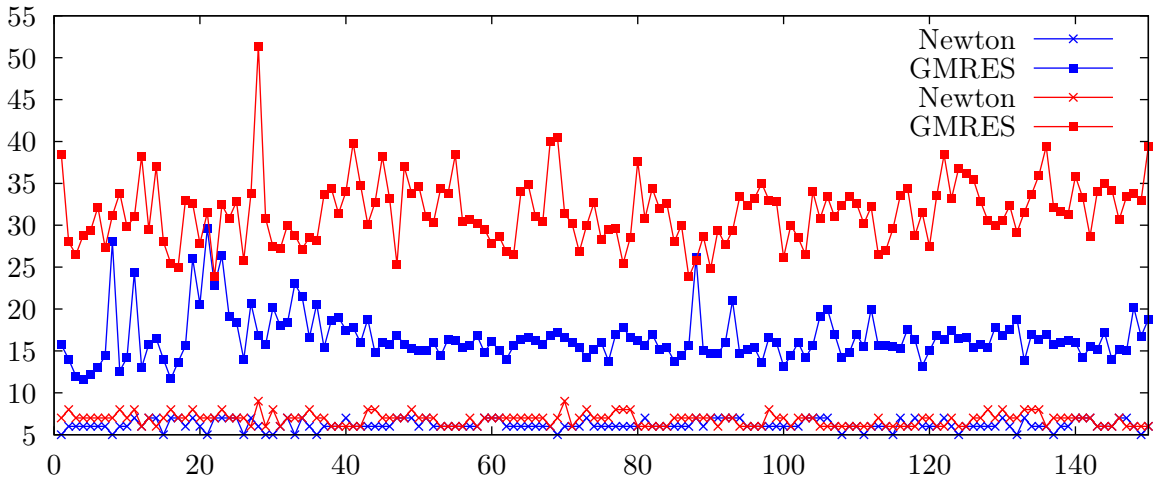

		\centering
		\include{fsi_richter_iter_inc}
		\caption{Number of iterations (Newton, average GMRES) required for the $3d$ test problem (blue: $r = 1$, red $r = 2$).}
		\label{fig:newton_iter}
	\end{figure}

	As we can see in Figure \ref{fig:scala:gmres3d},
	the parallel performance in the $3d$ case yields similar results as
	in the case of the 2d benchmark problem FSI-2.
	First, these findings show that our code
	is dimension-independent and can be employed for 2d and 3d simulations.
	Second, the obtained scalability is again in the range of
	$\mathcal{O}(\sqrt{n})$. Due to the higher computational cost in $3d$,
	tests were done on $16$ cores upwards. Specifically, for
	$16$ cores, the total CPU time to solve the linear problem with $14 \cdot 10^6$ dofs at a single time step
	is $2605$ seconds, i.e., approximately $43$ minutes. On $256$ cores, the
	computational cost decreases to $431$ seconds, i.e., $7.2$ minutes. Thus we
	achieve a reduction by $84\%$ of the computational time.

	\begin{figure}[h!]
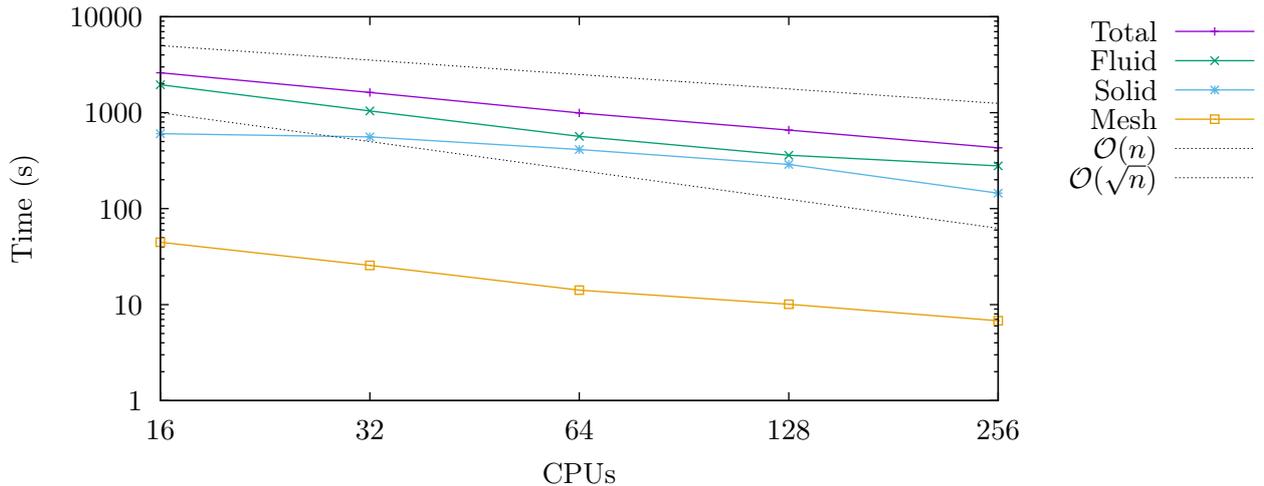

		\center
		\include{scala_3d_inc}
		\caption{Example 2: Strong scalability using the preconditioned GMRES for approximately $14 \cdot 10^6$ dofs in $3d$. Average time given in seconds for the solution of one linear system.}
		\label{fig:scala:gmres3d}
	\end{figure}

	\section{Conclusions}
	In this work, we developed a framework for the parallel solution of
	monolithic fluid-structure interaction (FSI) problems. The FSI problems
	is formulated with the help of the arbitrary Lagrangian-Eulerian technique.
	To cope with large solid deformations, we adopted a nonlinear harmonic
	mesh motion model. The key goals have been on the development of
	approximate block-LDU preconditioners in which we used Schur complement
	arguments.
	The parallel implementation is based on a combination of different
	software packages, which are mainly joined in the C++ package deal.II.
	To date only very few other studies have been published with
	satisfying results for FSI problems using high performance parallel computing
	and showing satisfactory scalability. This has been an important motivation of this
	work. Indeed, block-wise preconditioners perform pretty
	well in the serial case, but often lack perfect parallel scalability. This
	confirms the findings in \cite{ParallelFSI,ParallelFSI2}.
	Additionally, this type of preconditioner may be applied to any coupled
	problem exhibiting a $3 \times 3$ block-structure similar to the FSI system,
	provided that solvers for the subproblems are available. A similar
	preconditioner is employed in \cite{FSI:AMG},
	where additionally such a block-wise strategy is used as a smoother
	inside an AMG method. However, no scalability results have been reported therein.
	In our numerical tests, we have provided detailed studies for
	a challenging 2D benchmark problem and 3D test case. Both configurations
	are time-dependent and exhibit large solid deformations. In view of these
	aspects, the outcome of our results is more than satisfying.

	\section{Acknowledgments}

	This work has been supported by the Austrian Science Fund (FWF) grant No.
	P-29181
	`Goal-Oriented Error Control for Phase-Field Fracture Coupled to Multiphysics Problems'.

	\newpage

	\bibliographystyle{abbrv}
	\bibliography{literature}

\end{document}